\documentclass[10pt]{amsart}
\usepackage[margin=1.2in]{geometry}
\newtheorem{theorem}{Theorem}[section]
\newtheorem{lemma}[theorem]{Lemma}

\newtheorem{corollary}{Corollary}[section]
\theoremstyle{definition}

\theoremstyle{remark}

\numberwithin{equation}{section}

\begin{document}

\title[Polyatomic ellipsoidal BGK model]{Entropy production estimates for the polyatomic ellipsoidal BGK model}

\author{Sa Jun Park}
\address{Department of Mathematics, Sungkyunkwan University, Suwon 440-746, Republic of Korea}
\email{parksajune@skku.edu, }
\author{ SEOK-BAE YUN }
\address{Department of Mathematics, Sungkyunkwan University, Suwon 440-746, Republic of Korea}
\email{sbyun01@skku.edu}



\keywords{Ellipsoidal BGK model, Polyatomic gases, Boltzmann equation, Kinetic theory of gases, Entropy production estimates}

\begin{abstract}
We study the entropy production estimate for the polyatomic ellipsoidal BGK model, which is a
relaxation type kinetic model describing the time evolution of polyatomic particle systems. An interesting dichotomy is observed
between $0<\theta\leq 1$ and $\theta=0$:
In each case, a distinct target Maxwellians should be chosen to estimate the entropy production functional from below by the relative entropy.
The time asymptotic equilibrium state toward which the distribution function stabilizes bifurcates accordingly.
\end{abstract}
\maketitle
\section{introduction}
In this paper, we are interested in the entropy production property of the ellipsoidal BGK model for polyatomic molecules  \cite{ALPP,Brull3,Cai}
\begin{eqnarray}\label{ESBGK}
\partial_tf+v\cdot \nabla_xf&=&A_{\nu,\theta}(\mathcal{M}_{\nu,\theta}(f)-f),\cr
\qquad f(0,x,v,I)&=&f_0(x,v,I).
\end{eqnarray}
The polyatomic velocity distribution function $f(t,x,v,I)$ represents the number density on phase space $(x,v)\in \mathbb{R}^3_x\times \mathbb{R}^3_v$ with internal energy $I^{2/\delta}~(I\geq0)$ at time $t\geq0$.
Here, $\delta$ is the additional degree of freedom other than the translation motion. Such internal energy formulation can be traced back to \cite{Brun,KP,PL}.
The collision frequency $A_{\nu,\theta}$ is given by $A_{\nu,\theta}=(\rho T_{\delta})/\{\mu(1-\nu+\theta\nu)\}$
where $\mu>0$ denotes the viscosity.
 To explain the polyatomic ellipsoidal Gaussian
$\mathcal{M}_{\nu,\theta}(f)$,  we need to introduce several macroscopic quantities. We start with the definition of local density, bulk velocity, stress tensor and specific internal energy:
\begin{eqnarray*}
&&\rho(t,x)=\int_{\mathbb{R}^3 \times \mathbb{R}^+} f(t,x,v,I) dvdI, ~
U(t,x)=\frac{1}{\rho}\int_{\mathbb{R}^3 \times \mathbb{R}^+} vf(t,x,v,I) dvdI, \cr
&&\qquad\Theta(t,x) = \frac{1}{\rho}\int_{\mathbb{R}^3 \times \mathbb{R}^+} (v-U)\otimes(v-U)f(t,x,v,I) dvdI, \cr
&&\qquad E_{\delta}(t,x)=\frac{1}{\rho}\int_{\mathbb{R}^3 \times \mathbb{R}^+} \left(\frac{1}{2}|v-U|^2 +I^{\frac{2}{\delta}}\right)f(t,x,v,I) dvdI.
\end{eqnarray*}
The specific internal energy $E_{\delta}$ is divided into the energy from
the translational motion $E_{tr}$ and the energy due to the internal configuration $E_{int}$:
\begin{eqnarray*}
E_{tr}=\frac{1}{\rho}\int_{\mathbb{R}^3 \times \mathbb{R}^+} \frac{1}{2}|v-U|^2 f(t,x,v,I) dvdI, \quad E_{int} = \frac{1}{\rho}\int_{\mathbb{R}^3 \times \mathbb{R}^+}  I^{\frac{2}{\delta}} f(t,x,v,I) dvdI,
\end{eqnarray*}
which, as a consequence of equipartition theorem, are  associated  with the corresponding temperatures $T_{\delta}$, $T_{tr}$ and $T_{int}$ respectively:
\begin{eqnarray*}
E_{\delta}=\frac{3+\delta}{2}T_{\delta},\quad E_{tr}=\frac{3}{2}T_{tr}, \quad E_{int}= \frac{\delta}{2}T_{int}.
\end{eqnarray*}
Note that $T_{\delta}$ is represented by a convex combination of $T_{tr}$ and $T_{int}$:
\begin{eqnarray*}
T_{\delta}=\frac{3}{3+\delta}T_{tr}+\frac{\delta}{3+\delta}T_{int}.
\end{eqnarray*}
For $0\leq\theta\leq1$, we define the  relaxation  temperature $T_{\theta}$ and the corrected  temperature tensor $\mathcal{T}_{\nu,\theta}$ by
\begin{eqnarray*}
T_{\theta}=\theta T_{\delta}+(1-\theta)T_{int},\quad
\mathcal{T}_{\nu,\theta}=\theta T_{\delta}Id+(1-\theta)\big\{(1-\nu)T_{tr}Id+\nu \Theta\big\}.
\end{eqnarray*}
Now, the polyatomic ellipsoidal Gaussian $\mathcal{M}_{\nu,\theta}$ is given by
\begin{eqnarray*}
\mathcal{M}_{\nu,\theta}(f)= \frac{\rho \Lambda_{\delta}}{\sqrt{\det (2\pi \mathcal{T}_{\nu,\theta}}) T_{\theta}^{\frac{\delta}{2}}}\exp\left(-\frac{1}{2}(v-U)^{\top}\mathcal{T}^{-1}_{\nu,\theta}(v-U)-\frac{I^{\frac{2}{\delta}}}{T_{\theta}}\right).
\end{eqnarray*}
Here, $\Lambda_{\delta}$ denotes $\Lambda_{\delta}= 1/\int_{\mathbb{R}_+} e^{-I^{2/\delta}} dI$.
The relaxation operator satisfies the following cancellation property:
\begin{eqnarray*}
\int_{\mathbb{R}^3 \times \mathbb{R}^+} (\mathcal{M}_{\nu,\theta}(f)-f)
\left\{1,v,\frac{1}{2}|v|^2 +I^{\frac{2}{\delta}}\right\}
 dvdI = 0,
\end{eqnarray*}
which leads to the conservation of mass, momentum and energy. The $H$-theorem for this model was established in \cite{ALPP}
(See also \cite{Brull3}):
\[
\int_{\mathbb{R}^3 \times \mathbb{R}^+}f(t)\ln f(t)dvdI\leq
\int_{\mathbb{R}^3 \times \mathbb{R}^+}f_0\ln f_0dvdI,\quad (t\geq 0).
\]
\newline
The original BGK model \cite{BGK} for monatomic gases, which is widely used in place of the Boltzmann equation for practical purposes, has one well-known shortcoming that  it gives incorrect Prandtl number in the Navier Stokes limit. To overcome this, Holway \cite{Holway}  introduced a free parameter $\nu$ and generalized the local Maxwellian into the anistropic Gaussian, which is well-defined in the range $-1/2<\nu<1$ (See \cite{ALPP,Brull2,Yun2,Yun3}). The resultant model is called the ellipsoidal BGK model (ES-BGK model). In generalizing this model further to cover the polyatomic case, however, we are confronted with the another incorrect physical coefficient: the relaxation collision number, which is defined as the number of collision needed to transform the rotational and vibrational internal energy into the translational energy. In this regard, another  relaxation parameter $\theta$ is introduced (See \cite{ALPP,Brull3,Cai,Shen}), leading to the ellipsoidal BGK model for polyatomic particles (\ref{ESBGK}).

In this paper, we are concerned with the Cercignani type entropy-entropy production estimate for the polyatomic ellipsoidal BGK model (\ref{ESBGK}). Obtaining lower bounds of the entropy production functional for kinetic equations in terms of the relative entropy
is important in that such estimates provide the coercivity (at least partial) that pushes the distribution function to the equilibrium state.
It was first suggested by Cercignani \cite{Cercignani} for the Boltzmann equation, and culminated in \cite{V} where Villani proved the ``almost true" version of the conjecture. In their proof, the entropy production estimate of the Landau equation established in \cite{DV} was crucially used (See \cite{Wu} for recent improvement on this issue).
In \cite{Yun}, the author proved that the ellipsoidal BGK model for monatomic particle system (\cite{ALPP,Brull2,Holway}).
 satisfies the Cercignani type
entropy production estimate, implying that the entropy production mechanism of the ellipsoidal BGK model resembles that of
the linear Boltzmann equation, rather than that of the full Boltzmann equation. (See \cite{Bisi}).
In this paper, we extend the result to the polyatomic ellipsoidal BGK model (See Theorem 1.1 below).
Due to the presence of various types of temperatures in the polyatomic ellipsoidal Gaussian, the fine cancellation of the temperature function in the
entropy comparison of various Maxwellians, which was crucially used in the proof in \cite{Yun}, is not available in the polyatomic case,
and we need to keep track of the behavior of those temperatures carefully throughout the argument (See Lemma 2.1).

We also make an interesting observation that different target equilibrium states, to which the distribution function converges time asymptotically, should be chosen according to the value of $\theta$:
when $0<\theta\leq 1$, the relative entropy should be measured with respect to $\mathcal{M}_{0,1}$ where
\begin{eqnarray*}
\mathcal{M}_{0,1}= \frac{\rho \Lambda_{\delta}}{(2\pi T_{\delta})^{\frac{3}{2}} (T_{\delta})^{\frac{\delta}{2}}}\exp\left(-\frac{|v-U| ^2}{2T_{\delta}}-\frac{I^{\frac{2}{\delta}}}{T_{\delta}}\right),
\end{eqnarray*}
while it is $\mathcal{M}_{0,0}$ when $\theta=0$ for
\begin{eqnarray*}
\mathcal{M}_{0,0}= \frac{\rho \Lambda_{\delta}}{(2\pi T_{tr})^{\frac{3}{2}} (T_{int})^{\frac{\delta}{2}}}\exp\left(-\frac{|v-U| ^2}{2T_{tr}}-\frac{I^{\frac{2}{\delta}}}{T_{int}}\right).
\end{eqnarray*}
This is because, when $\theta=0$, the translational energy and the internal energy is split, making the equation essentially, but not exactly, monatomic. More precisely, the internal energy part is cancelled out in measuring the difference of $H$-functional of various Maxwellians ( See Lemma 3.1 in Section 3). This implies a dichotomy in the time asymptotic state of the distribution function $f$, namely,
$\mathcal{M}_{0,1}$ for $0<\theta\leq 1$ and $\mathcal{M}_{0,0}$ for $\theta=0$. We note that such a phenomena is
not observed in the monatomic case \cite{Yun}.\newline
%
%
%
%

Let us define the $H$-functional $H(f)$, the relative entropy $H(f|g)$ and the entropy production functional $D_{\nu,\theta}(f)$:
\begin{eqnarray*}
&&H(f)=\int_{\mathbb{R}^3\times \mathbb{R}^+} f\ln f dvdI, \quad H(f|g)=\int_{\mathbb{R}^3\times \mathbb{R}^+} f\ln(f/g) dvdI,\cr
&&\qquad D_{\nu,\theta}(f) = -\int_{\mathbb{R}^3\times \mathbb{R}^+} A_{\nu,\theta}\big\{\mathcal{M}_{\nu,\theta}(f)-f\big\}\ln f dvdI.
\end{eqnarray*}
Our main result is as follows:
\begin{theorem}\label{main result}
For $0\leq\theta\leq 1$ and $-1/2<\nu<1$, the entropy production functional $D_{\nu,\theta}(f)$ of the ES-BGK model satisfies
\begin{enumerate}
\item In the case $0<\theta\leq 1$
\begin{eqnarray*}
D_{\nu,\theta}(f)\geq\theta A_{\nu,\theta} H(f|\mathcal{M}_{0,1}).
\end{eqnarray*}
\item In the case $\theta=0$
\begin{eqnarray*}
D_{\nu,0}(f)\geq\min\{1-\nu, 1+2\nu\} A_{\nu,0} H(f|\mathcal{M}_{0,0}).
\end{eqnarray*}
\end{enumerate}
\end{theorem}
These entropy production estimates readily give the asymptotic behavior of $f$ in the homogeneous case:
\begin{corollary} The distribution function for the spatially homogeneous polyatomic ellipsoidal BGK model stabilizes
exponentially fast to the equilibrium states:
\begin{enumerate}
\item In the case $0<\theta\leq 1$
\begin{eqnarray*}
\|f(t)-\mathcal{M}_{0,1}\|_{L^1_{v,I}}\leq e^{-\frac{\theta}{2}A_{\nu,\theta}t }
\sqrt{2H(f_0|\mathcal{M}_{0,1})}.
\end{eqnarray*}
\item  In the case $\theta=0$
\begin{eqnarray*}
\|f(t)-\mathcal{M}_{0,0}\|_{L^1_{v,I}}\leq e^{-\frac{1}{2}A_{\nu,0}\min\{1-\nu,1+2\nu\}t}
\sqrt{2H(f_0|\mathcal{M}_{0,0})}.
\end{eqnarray*}
\end{enumerate}
where $\|f(t)\|_{L^1_{v,I}}=\int_{\mathbb{R}^3_v\times \mathbb{R}_+}|f(v,t,I)|dvdI$.
\end{corollary}
Some remarks are in order. First, these results are a priori estimates, which means that they hold when everything is fine: For this to be mathematically rigorous,  integrability of the distribution function should be good enough to justify all the integral in the proof, and the strict positivity of the temperatures $T_{int}$, $T_{tr}$ should be assumed, which should be checked at the level of existence theory.
 These issues were checked  for  monatomic ES-BGK model in \cite{Yun2,Yun}. The investigation on the existence theory for the polyatomic case is in progress. Secondly, these results can be generalized in a straightforward manner to general $d$-dimensions. For this, the temperatures should be redefined as
$E_{\delta}=\frac{d+\delta}{2}T_{\delta}$, $ E_{tr}=\frac{d}{2}T_{tr}$ and the constant $3$ in several places, for example in (\ref{another use}) and (\ref{Hfunc}) should be replaced by $d$. The argument then goes in the exactly same manner, giving the essentially same result with the constants adjusted according to $d$. Instead of treating the most general case, however, we restrict ourselves to three dimensional case for clarity of the proof.\newline

This paper is organized as follows. In section 2, we prove the entropy production estimate in the case $0<\theta\leq1$.
It is also shown that the spatially homogeneous distribution function converges exponentially fast to $\mathcal{M}_{0,1}$.
In section 3, analogous result is proved for the case $\theta=0$, with the target Maxwellian replaced by $\mathcal{M}_{0,0}$.
%
%
%
%
\section{Entropy production estimate in the case: $0<\theta\leq1$}
We need to introduce the following multi-variate Gaussian with the stress tensor as its covariance matrix, which plays an important role in the proof of our main theorem:
\begin{eqnarray*}
\mathcal{M}_{\Theta}(f)= \frac{\rho \Lambda_{\delta}}{\sqrt{\det 2\pi \Theta} (T_{int})^{\frac{\delta}{2}}}\exp\left(-\frac{1}{2}(v-U)^{\top}\Theta^{-1}(v-U)-\frac{I^{\frac{2}{\delta}}}{T_{int}}\right).
\end{eqnarray*}
Note that $\mathcal{M}_{\Theta}$ corresponds to $\mathcal{M}_{1,0}$.
We start with the following lemma connecting the $H$-functionals of $\mathcal{M}_{\nu,\theta}$, $\mathcal{M}_{\Theta}$ and $\mathcal{M}_{0,1}$.
\begin{lemma} The $H$-functionals for $\mathcal{M}_{0,1}$, $\mathcal{M}_{\Theta}$ and $\mathcal{M}_{\nu,\theta}$ satisfy
\begin{eqnarray*}
H(\mathcal{M}_{0,1})-H(\mathcal{M}_{\nu,\theta}) \geq (1-\theta) \{H(\mathcal{M}_{0,1})-H(\mathcal{M}_{\Theta})\},~ (0<\theta\leq1)
\end{eqnarray*}
\end{lemma}
\begin{proof}
A straightforward calculation gives
\begin{eqnarray}\label{Hfunc}
\begin{split}
H(\mathcal{M}_{\nu,\theta})&= \rho\ln\rho\Lambda_{\delta}-\frac{1}{2}\rho\ln(\det(2\pi\mathcal{T}_{\nu,\theta}))-\frac{\delta}{2}\rho\ln T_{\theta}-\frac{3+\delta}{2}\rho, \cr
H(\mathcal{M}_{\Theta})&= \rho\ln\rho\Lambda_{\delta}-\frac{1}{2}\rho\ln(\det(2\pi\Theta))-\frac{\delta}{2}\rho\ln T_{int}-\frac{3+\delta}{2}\rho, \cr
H(\mathcal{M}_{0,1})&= \rho\ln\rho\Lambda_{\delta}-\frac{3}{2}\rho\ln(2\pi T_{\delta})-\frac{\delta}{2}\rho\ln T_{\delta}-\frac{3+\delta}{2}\rho,
\end{split}
\end{eqnarray}
so that
\begin{eqnarray}\label{Max-Gau}
H(\mathcal{M}_{0,1})-H(\mathcal{M}_{\nu,\theta}) = \frac{\rho}{2}\big\{\ln (\det \mathcal{T}_{\nu,\theta})+\delta \ln T_{\theta}-(3+\delta)\ln T_{\delta}\big\}
\equiv I_{\delta,\theta}.
\end{eqnarray}
Due to the symmetry of $\Theta$, there exists an orthogonal matrix $P$ such that $P^{\top}\Theta P $ is a diagonal matrix. We denote its eigenvalues by
$\Theta_i$ $(i=1,2,3)$ to compute
\begin{align*}
\det\mathcal{T}_{\nu,\theta}&=\det\Big\{P^{\top}\mathcal{T}_{\nu,\theta} P\Big\}\cr
 &=  \det\Big\{(1-\theta)\big\{(1-\nu)T_{tr}Id+\nu P^{\top}\Theta P\big\}+\theta T_{\delta}Id\Big\}\cr
&= \prod_{1\leq i\leq 3}\big\{(1-\theta)\big((1-\nu)T_{tr}+\nu\Theta_i\big)+\theta T_{\delta}\big\}.
\end{align*}
Hence,
\begin{eqnarray*}
\ln\det\mathcal{T}_{\nu,\theta}
=\sum_{1\leq i\leq 3} \ln\big\{(1-\theta)(1-\nu)T_{tr}+(1-\theta)\nu\Theta_i+\theta T_{\delta}\big\}.
\end{eqnarray*}
Now, we divide the remaining argument into the following two cases:\newline

\noindent$(1)$ $0\leq\nu<1$: Recalling the concavity of $\ln$, we have
\begin{align*}
\ln\det\mathcal{T}_{\nu,\theta}
&\geq\sum_{1\leq i\leq 3} \big\{(1-\theta)(1-\nu)\ln T_{tr}+(1-\theta)\nu \ln\Theta_i+\theta \ln T_{\delta}\big\}\cr
 &= 3(1-\theta)(1-\nu)\ln T_{tr} +(1-\theta)\nu\ln \Theta_1\Theta_2\Theta_3+3\theta \ln T_{\delta}\cr
 &= 3(1-\theta)(1-\nu)\ln T_{tr} +
(1-\theta)\nu\ln\det\Theta+3\theta \ln T_{\delta}.
\end{align*}
Inserting this, we estimate (\ref{Max-Gau}) as
\begin{align*}
I_{\delta,\theta}&\geq  \frac{\rho}{2}\Big\{3(1-\theta)(1-\nu)\ln T_{tr} +
(1-\theta)\nu\ln\det\Theta+3\theta \ln T_{\delta} +\delta\ln T_{\theta}-(3+\delta)\ln T_{\delta}\Big\}  \cr
&= \frac{\rho}{2}\Big\{~(1-\theta)\big(3(1-\nu)\ln T_{tr} +
\nu\ln\det\Theta-3 \ln T_{\delta}\big) +\delta\ln T_{\theta}-\delta\ln T_{\delta}\Big\}.
\end{align*}
We then employ $\ln T_{\theta} \geq (1-\theta)\ln T_{int} +\theta \ln T_{\delta}$ to see that
\begin{eqnarray}\label{to see}
\delta\ln T_{\theta}-\delta\ln T_{\delta}
\geq\delta(1-\theta)\big\{\ln T_{int}-\ln T_{\delta}\big\}
\end{eqnarray}
from which we get
\begin{eqnarray*}
I_{\delta,\theta}
\geq  (1-\theta)\frac{\rho}{2}\big\{ 3(1-\nu)\ln T_{tr} +
\nu\ln\det\Theta-3 \ln T_{\delta}+\delta\big(\ln T_{int}-\ln T_{\delta}\big)\Big\}.
\end{eqnarray*}
Then, in regard of the following relation between $T_{tr}$ and $\Theta$ \cite{ALPP}:
\begin{eqnarray}\label{another use}
3\ln T_{tr}=\ln\Big(\frac{\Theta_1+\Theta_2+\Theta_3}{3}\Big)^3 \geq \ln \Theta_1\Theta_2\Theta_3 = \ln\det \Theta,
\end{eqnarray}
which is a direct consequence of arithmetic-geometric inequality, we can proceed further as
\begin{align*}
I_{\delta,\theta}&\geq
(1-\theta)\frac{\rho}{2}\big\{ (1-\nu)\ln \det\Theta + \nu\ln\det\Theta-3 \ln T_{\delta}+\delta\big(\ln T_{int}-\ln T_{\delta}\big)\big\}\cr
&=(1-\theta)\frac{\rho}{2}\big\{\ln\det\Theta+\delta\ln T_{int}-(3+\delta)\ln T_{\delta}\big\}.
\end{align*}
Another explicit computation using (\ref{Hfunc}) shows that this is exactly
$(1-\theta)\{H(\mathcal{M}_{0,1})-H(\mathcal{M}_{\Theta})\}$,
which gives the desired estimate for positive $\nu$.\newline

\noindent (2) $-1/2 < \nu \leq 0$: In this case, $(1-\nu)T_{tr}+\nu\Theta_i $ is not a convex combination of $T_{tr}$ and $\Theta_i$. Instead, we use $\Theta_1 +\Theta_2 +\Theta_3 =3T_{tr} $
to see
\[
(1-\nu)T_{tr}+\nu\Theta_i=(1+2\nu)T_{tr}-\nu\sum_{j\neq i}\Theta_j
\]
so that
\begin{eqnarray*}
\det\mathcal{T}_{\nu,\theta}=\prod_{1\leq i\leq3}\Big\{(1-\theta)\Big((1+2\nu)T_{tr}-\nu \sum_{j\neq i}\Theta_j\Big)+\theta T_{\delta}\Big\}.
\end{eqnarray*}
Taking $\ln$ on both sides and using the concavity inequality, we get
\begin{eqnarray*}
\ln\det\mathcal{T}_{\nu,\theta}&=&
\sum_{1\leq i\leq 3}\ln\Big\{(1-\theta)\Big((1+2\nu)T_{tr}-\nu \sum_{j\neq i}\Theta_j\Big)+\theta T_{\delta}\Big\} \cr
&\geq&\sum_{1\leq i\leq 3}\Big\{(1-\theta)\Big((1+2\nu)\ln T_{tr}-\nu \sum_{j\neq i}\ln\Theta_j\Big)+\theta \ln T_{\delta}\Big\} \cr
&=& (1-\theta)\big\{3(1+2\nu)\ln T_{tr}-2\nu\ln\det\Theta\big\} +3\theta\ln T_{\delta}.
\end{eqnarray*}
Now, we can compute similarly as in the previous case as
\begin{eqnarray*}
\begin{split}
& \ln \det \mathcal{T}_{\nu,\theta} +\delta\ln T_{\theta}-(3+\delta)\ln T_{\delta} \cr
&\qquad\geq \Big\{(1-\theta)\Big(3(1+2\nu)\ln T_{tr}-2\nu\ln\det\Theta\Big) +3\theta\ln T_{\delta}\Big\} +\delta\ln T_{\theta}-(3+\delta)\ln T_{\delta}\cr
&\qquad= (1-\theta)\big\{3(1+2\nu)\ln T_{tr}-2\nu\ln\det\Theta-3\ln T_{\delta}\big\} +\delta\ln T_{\theta}-\delta\ln T_{\delta}.
\end{split}
\end{eqnarray*}
We recall (\ref{to see}) to bound the last line from below by
\begin{eqnarray*}
&&(1-\theta)\big\{3(1+2\nu)\ln T_{tr}-2\nu\ln\det\Theta-3\ln T_{\delta}\big\} +\delta(1-\theta)\big\{\ln T_{int}-\ln T_{\delta}\big\} \cr
&&\qquad= (1-\theta) \Big\{\Big(3(1+2\nu)\ln T_{tr}-2\nu\ln\det\Theta-3\ln T_{\delta}\Big)
+\delta\big(\ln T_{int}-\ln T_{\delta}\big)\Big\}\cr
&&\qquad= (1-\theta) \big\{3(1+2\nu)\ln T_{tr}-2\nu\ln\det\Theta  +\delta\ln T_{int}-(3+\delta)\ln T_{\delta}\big\}.
\end{eqnarray*}
Therefore, by making another use of (\ref{another use}), we obtain
\begin{eqnarray*}
I_{\delta,\theta}
\geq (1-\theta)\frac{\rho}{2}\big\{\ln \det\Theta+\delta\ln T_{int})-(3+\delta)\ln T_{\delta}\big\},
\end{eqnarray*}
which, again from (\ref{Hfunc}), can be shown to be
$(1-\theta)\left\{H(\mathcal{M}_{0,1})-H(\mathcal{M}_{\Theta})\right\}$.
This completes the proof.
\end{proof}
The following lemma can be found in \cite{ALPP}.
\begin{lemma} \cite{ALPP}
 The H-functionals of the f, $\mathcal{M}_{\Theta}$ and $\mathcal{M}_{0,1}$ are related by
\begin{eqnarray*}
H(\mathcal{M}_{0,1}) \leq H(\mathcal{M}_{\Theta}) \leq H(f).
\end{eqnarray*}
\end{lemma}

%
%
%
%
\subsection{Proof of Theorem \ref{main result} (1)}
We first recall $F^{\prime}(x)(x-y)\geq F(x)-F(y)$
satisfied by any convex function $F$, which, in view of the convexity of $x\ln x$ implies
\begin{align*}
D_{\nu,\theta}(f)&=-\int_{\mathbb{R}^3\times \mathbb{R}^+} A_{\nu,\theta}\big\{\mathcal{M}_{\nu,\theta}(f)-f\big\}\ln f dvdI\cr
&=A_{\nu,\theta}\int_{\mathbb{R}^3\times \mathbb{R}^+} \big\{f-\mathcal{M}_{\nu,\theta}(f)\big\}H^{\prime}(f)vdI\cr
&\geq A_{\nu,\theta}\big\{ H(f)-H(\mathcal{M}_{\nu,\theta})\big\}.
\end{align*}
We divide the last term as
\begin{eqnarray*}
H(f)-H\big(\mathcal{M}_{\nu,\theta}\big)=H(f)-H(\mathcal{M}_{0,1})+H(\mathcal{M}_{0,1})-H(\mathcal{M}_{\nu,\theta}).
\end{eqnarray*}
Then, by Lemmas 2.1 and 2.2, we obtain
\begin{align*}
H(f)-H(\mathcal{M}_{\nu,\theta})&\geq H(f|\mathcal{M}_{0,1})+(1-\theta)\{H(\mathcal{M}_{0,1})-H(\mathcal{M}_{\Theta})\} \cr
&\geq H(f|\mathcal{M}_{0,1})+(1-\theta)\{H(\mathcal{M}_{0,1})-H(f)\} \cr
&= H(f|\mathcal{M}_{0,1})+(\theta-1) H(f|\mathcal{M}_{0,1}) \cr
&= \theta H(f|\mathcal{M}_{0,1}).
\end{align*}
\subsection{The proof of Corollary 1.1 (1)}
From Theorem 1.1 (1), we get
\begin{eqnarray*}
\frac{d}{dt}H(f|\mathcal{M}_{0,1})=-D_{\nu,\theta}(f)
\leq-\theta A_{\nu,\theta}H(f|\mathcal{M}_{0,1}).
\end{eqnarray*}
Note that  $A_{\nu,\theta}$ is a constant since $\rho$ and $T_{\delta}$ is constant in the homogeneous case. Then Gronwall's lemma gives
\begin{eqnarray*}
H(f|\mathcal{M}_{0,1})\leq e^{-\theta A_{\nu,\theta}t}H(f_0|\mathcal{M}_{0,1}).
\end{eqnarray*}
Hence, the application of the Kullback inequality:
\begin{eqnarray*}
\|f-g\|_{L^1}\leq \sqrt{2H(f|g)}\quad \mbox{ if }\quad \int f=\int g
\end{eqnarray*}
gives the desired result.
%
%
%
%
\section{Entropy production estimate in the case: $\theta=0$ }
In this case, we see that $T_{\theta} = T_{int}$ and $ \mathcal{T}_{\nu,0}=(1-\nu)T_{tr}Id+\nu\Theta$
to get
\begin{eqnarray*}
\mathcal{M}_{\nu,0}(f)=\frac{\rho \Lambda_{\delta}}{\sqrt{\det (2\pi \mathcal{T}_{\nu,0})} ( T_{int})^{\frac{\delta}{2}}}\exp\left(-\frac{1}{2}(v-U)^{\top}\mathcal{T}_{\nu,0}^{-1}(v-U)-\frac{I^{\frac{2}{\delta}}}{ T_{int}}\right).
\end{eqnarray*}
\begin{lemma}The H-functional for $\mathcal{M}_{0,0}$, $\mathcal{M}_{\Theta}$ and $\mathcal{M}_{\nu,0}$ satisfies
\begin{eqnarray*}
H(\mathcal{M}_{0,0})-H(\mathcal{M}_{\nu,0}) \geq
\max\{\nu,-2\nu\}\{H(\mathcal{M}_{0,0})-H(\mathcal{M}_{\Theta})\}.
\end{eqnarray*}
\end{lemma}
\begin{proof}
An explicit computation, which is almost identical with the one given in \cite{Yun} gives
\begin{eqnarray*}
&&H(\mathcal{M}_{0,0})-H(\mathcal{M}_{\nu,0})\cr
&&\qquad=\left\{
    \begin{array}{ll}
      (1/2)\rho\big\{3(1-\nu)\ln T_{tr}+\nu\ln\det\Theta -3\ln T_{tr}\big\} & \hbox{for }~0\leq \nu <1, \\
      (1/2)\rho\big\{3(1+2\nu)\ln T_{tr}-2\nu\ln\det\Theta -3\ln T_{tr}\big\} & \hbox{for}-1/2< \nu \leq0.
    \end{array}
  \right.
\end{eqnarray*}
Note that $T_{int}$ is cancelled out in this case, which reduces the remaining computation to that of the  monatomic case
carried out in \cite{Yun}. Therefore, by Lemma 2.2 in \cite{Yun}, we get the desired result.
\end{proof}
\subsection{Proof of Theorem 1.1 (2)}:
As in the previous case, we have from the convexity of $x\ln x$:
\begin{eqnarray*}
D_{\nu,0}(f)\geq A_{\nu,0}\left\{H(f)-H(\mathcal{M}_{\nu,0})\right\}.
\end{eqnarray*}
We split the last term as
\begin{eqnarray*}
H(f)-H(\mathcal{M}_{\nu,0})=H(f)-H(\mathcal{M}_{0,0})+H(\mathcal{M}_{0,0})-H(\mathcal{M}_{\nu,0})
\end{eqnarray*}
and apply Lemma 3.1 to get the desired result:
\begin{align*}
H(f)-H(\mathcal{M}_{\nu,0})&\geq H(f|\mathcal{M}_{0,0})+\max\{\nu,-2\nu\}\{H(\mathcal{M}_{0,0})-H(\mathcal{M}_{\Theta})\} \cr
&\geq H(f|\mathcal{M}_{0,0})+\max\{\nu,-2\nu\}\{H(\mathcal{M}_{0,0})-H(f)\} \cr
&= H(f|\mathcal{M}_{0,0})-\max\{\nu,-2\nu\} H(f|\mathcal{M}_{0,0}) \cr
&= \min\big\{1-\nu,1+2\nu\big\} H(f|\mathcal{M}_{0,0}),
\end{align*}
where we used $H(\mathcal{M}_{\Theta}) \leq  H(f)$.\newline
The proof for the Corollary 1.1 (2) is identical to the previous case. we omit it.
\newline

\noindent{\bf Acknowledgement} This research was supported by Basic Science Research Program through
the National Research Foundation of Korea (NRF) funded by the Ministry of Science, ICT $\&$
Future Planning (NRF-2014R1A1A1006432)


%
%
\bibliographystyle{amsplain}

\end{document}